\documentclass[12pt]{article}
\usepackage{latexsym}
\usepackage{amssymb}
\usepackage{amsmath}
\usepackage{amsthm}
\usepackage{mathrsfs}
\usepackage{wasysym}
\usepackage[pdftex]{graphicx}
\usepackage{enumerate}
\usepackage{rawfonts}
\input{prepictex}
\input{pictex}
\input{postpictex}
\numberwithin{equation}{section}
\usepackage[OT2,OT1]{fontenc}
\def\cyr{%
\renewcommand\rmdefault{wncyr}%
\renewcommand\sfdefault{wncyss}%
\renewcommand\encodingdefault{OT2}%
\normalfont
\selectfont}
\DeclareMathAlphabet{\zap}{OT1}{pzc}{m}{it}
\DeclareTextFontCommand{\textcyr}{\cyr}
\def\be{\begin{equation}}
\def\ee{\end{equation}}
\def\bea{\begin{eqnarray*}}
\def\eea{\end{eqnarray*}}

\newcommand{\rad}{\text{\cyr   ya}}
\newcommand{\eg}{\text{\cyr   \bf G}}

\def\CC{\mathbb{C}}

\DeclareMathOperator{\Vol}{Vol}

\newtheorem{thm}{Theorem}[section]

\newtheorem{prop}[thm]{Proposition}
\newtheorem{cor}[thm]{Corollary}
\newtheorem{defn}[thm]{Definition}

\newenvironment{rmk}{\mbox{ }\\{\bf  Remark}\mbox{ }}{
\hfill  $\diamondsuit$\mbox{}\bigskip}
\def\ZZ{{\mathbb{Z}}}
\def\RR{{\mathbb{R}}}
\def\CP{{\mathbb{C} \mathbb{P}}}
\begin{document}

\title{Mass,  K\"ahler Manifolds, and\\ Symplectic Geometry}

\author{Claude LeBrun\thanks{Research funded in part by  the Simons Foundation and the \'Ecole Normale Sup\'erieure.}}

\date{}
\maketitle

\begin{abstract}
In the author's previous joint work  with Hans-Joachim Hein \cite{heinleb},  a mass formula 
for asymptotically locally Euclidean (ALE) K\"ahler manifolds was proved, assuming only 
relatively weak   fall-off conditions on the metric. However, the  
case of real dimension $4$ presented  technical difficulties that led us to require fall-off conditions
in this special dimension that are  stronger than the   Chru\'{s}ciel
fall-off conditions that sufficed  in higher dimensions. The present article, however,  
shows that techniques of $4$-dimensional symplectic geometry can be used to 
obtain   all the major results of \cite{heinleb}, assuming only Chru\'{s}ciel-type fall-off.
In particular, the present  article presents  a new a  proof of our Penrose-type inequality for the mass of
an asymptotically Euclidean K\"ahler manifold that only requires  Chru\'{s}ciel metric fall-off. 
 \end{abstract}

A complete connected non-compact  
Riemannian manifold $(M,g)$ of real dimension $n\geq 3$  is said to be {\em asymptotically Euclidean} (or {\em AE}\,) if there is 
a compact subset $\mathbf{K}\subset M$ such that $M-\mathbf{K}$ consists of finitely many components, each of which 
is diffeomorphic to the complement of a closed ball $\mathbf{D}^n \subset \RR^n$  in  such a manner  
that $g$ becomes the standard Euclidean metric plus terms that fall off sufficiently rapidly at infinity. 
More generally, a Riemannian $n$-manifold $(M,g)$ is said to be {\em asymptotically locally Euclidean} (or {\em ALE}\,) 
if the complement of a compact set $\mathbf{K}$ consists of finitely many components, each of which 
is diffeomorphic to  a quotient $(\RR^n-\mathbf{D}^n)/\varGamma_i$, where    $\varGamma_i\subset {\mathbf O}(n)$ is a finite subgroup 
that acts freely on the unit sphere,
in such a way that $g$ again becomes the Euclidean metric plus error terms that fall off sufficiently rapidly at infinity.
The components of $M-\mathbf{K}$ are called the {\em ends} of $M$;  their fundamental groups are 
the aforementioned groups $\varGamma_i$, and might in principle  be different
for different ends of the manifold. 

There is no clear consensus regarding the  exact fall-off conditions that should be imposed on the metric $g$,
as various authors have in practice tweaked the definition to dovetail  with the technical requirements  demanded by their favorite techniques. However, the weakest standard 
hypotheses that seem to lead to compelling results are the ones  introduced by Chru\'{s}ciel \cite{admchrus}:
\label{conditions}
 \begin{enumerate}[(i)]
 \item
 \label{eins} The metric  $g$ is of class $C^2$, with scalar curvature $s$ in  $L^1$;  and
\item
\label{zwei}
 in  some asymptotic chart at each end of $M^n$, and for some $\varepsilon > 0$, the components of the metric 
satisfy $$
g_{jk} = \delta_{jk} + O (|x|^{1-\frac{n}{2}-\varepsilon}), \qquad g_{jk,\ell} = 
O (|x|^{-\frac{n}{2}-\varepsilon}). 
$$ 
 \end{enumerate}
 With these very weak hypotheses, Chru\'{s}ciel's argument shows that the {\em mass} 
 $$
{\zap m}(M, g) := \lim_{\varrho\to \infty}  
 \frac{\mathbf{\eg} (\frac{n}{2})}{4(n-1)\pi^{n/2}} \int_{S_\varrho/\varGamma_i} \left[ g_{k\ell, k} -g_{kk,\ell}\right] \mathbf{n}^\ell d\mathfrak{a}_E
$$
of   an ALE manifold $(M,g)$ at any given end 
is both well-defined and invariant under a large class of changes of asymptotic coordinate system;
here, commas indicate  partial derivatives in the given asymptotic coordinates, 
summation over
repeated indices is understood, 
$S_\varrho$ is the Euclidean coordinate sphere of radius $\varrho$,  $\varGamma_i$ is the fundamental group of the relevant end, 
$\mathbf{\eg}$ is the Euler Gamma function, 
$d\mathfrak{a}_E$ is the $(n-1)$-dimensional volume form induced on this sphere by the Euclidean 
metric,  and $\vec{\mathbf{n}}$ is the outward-pointing Euclidean unit normal vector.
While  Chru\'{s}ciel's   paper actually only discusses the AE case,  his argument immediately extends to the more general ALE 
setting under discussion  here. 
The fact that fall-off conditions on the metric are by no means a matter of widespread consensus is nicely illustrated by 
 Bartnik's  powerful and better-known  theorem   \cite{bartnik} on the coordinate-invariance of the mass, which was  proved  around the same time
as Chru\'{s}ciel's work;  while  Bartnik's   
conclusion regarding the  coordinate invariance of the mass is markedly stronger than Chru\'{s}ciel's, it  is obtained at the price of replacing  hypothesis \eqref{zwei}  above with the 
stronger
assumption that the $g_{jk}-\delta_{jk}$ belong to a weighted Sobolev spaces $W^{2,q}_{-\tau}$ for some $q > n$ and some $\tau > (n-2)/2$.
Bartnik's paper is also notable for showing,  by counter-example, that any significant weakening of Chru\'{s}ciel's conditions \eqref{eins} and \eqref{zwei} would
result in the mass being ill-defined and/or  coordinate dependent.

In joint work with H.-J. Hein, the present author has  elsewhere shown \cite{heinleb} that if $(M, g, J)$ is an ALE {\em K\"ahler} manifold of complex dimension $m$, 
then $M$ has only one end, and that the  mass  at this unique end  is  given by 
$${\zap m}(M,g) =  {\textstyle \frac{1}{(2m-1)\pi^{m-1}}} \langle \clubsuit (-c_1) , [\omega ]^{m-1}\rangle+
{\textstyle \frac{(m-1)!}{4(2m-1)\pi^m}} \int_M s_g d\mu_g$$
where $s_g$ and $d\mu_g$ are respectively  the scalar curvature and volume form of $g$,  $c_1 =c_1^{\RR}(M,J)\in H^2 (M)$ is the first Chern class of the complex structure, 
$[\omega]\in H^2(M)$ is the
K\"ahler class of $g$, $\clubsuit : H^2(M) \to H^2_c (M)$ is the inverse of the natural morphism 
 from compactly supported to ordinary deRham cohomology, and 
 $\langle~, ~\rangle$ is the duality pairing between $H^2_c(M)$ and 
$H^{2m-2}(M)$. 
If one  accepts it as  {\em given} that $M$ has only one end, our proof \cite[\S 3]{heinleb}  of the above mass formula 
only requires the Chru\'{s}ciel fall-off hypotheses \eqref{eins} and \eqref{zwei},
and provides an entirely self-contained  proof of the  coordinate-invariance of the mass in the K\"ahler case. 
However,  our proof that $M$   {\em can} only  have one end,   merely assuming  the metric fall-off condition  \eqref{zwei}, works well only when 
 $m\geq 3$; when $m=2$,  our  proof only managed to  obtain the same conclusion from Chru\'{s}ciel's  fall-off hypothesis if 
 $\varepsilon > 1/2$. This and related phenomena led  us, in \cite{heinleb},  to instead  insist on  Bartnik-type metric fall-off in the special case real of dimension $4$. 
 
This note will provide a remedy for this  state of affairs. Many of the analytic subtleties encountered 
in the $4$-dimensional are subtly intertwined with the fact that the complex structure of an ALE K\"ahler surface need not be standard at infinity. 
By contrast, we will show here that  the {\em symplectic structure}  at infinity of such a manifold {\em is}
always standard, even with  extremely weak  fall-off assumptions on   the metric. 
By  developing  symplectic versions  of some of the previous  proofs,  
we will thus be able  to show that, even when $m=2$,    all the the main results of 
\cite{heinleb}   continue to  hold even when the metric simply satisfies  Chru\'{s}ciel's
 weak  fall-off hypotheses \eqref{eins} and \eqref{zwei}. 
 In particular, we will see that our Penrose-type inequality \cite[Theorem E] {heinleb} for the mass of an AE K\"ahler manifold
remains valid  even in real dimension four, assuming only the  mildest reasonable  fall-off assumptions on the metric.

 \section{The Asymptotic Symplectic Structure}
 \label{logos}
 
 For  clarity and concreteness, we will restrict the following discussion to real dimension $4$. However,  most of what follows 
does  work, {\em mutatis mutandis}, in 
   higher dimensions, and indeed  is actually far  less delicate in that setting. 
 
 Let $(M^4, g, J)$ be a an ALE K\"ahler surface, which we hypothetically allow to perhaps have several ends. Throughout, we will simply  
 assume  that $g$ satisfies the Chru\'{s}ciel 
  fall-off hypothesis, and in this section we will actually only make use  of hypothesis \eqref{zwei}  with $n=4$. 
  Thus, on any given end $M_{\infty,i}$ of $M$, we assume that there are asymptotic coordinates $(x^1, \ldots, x^4)$
on the universal cover $\widetilde{M}_{\infty,i}$ of $M_{\infty,i}$  in which the components of the metric  satisfy 
$$
g_{jk} = \delta_{jk} + O (|x|^{-1-\varepsilon}), \qquad g_{jk,\ell} = 
O (|x|^{-2-\varepsilon})
$$ 
for some $\varepsilon > 0$,
and such that  the fundamental group $\varGamma_i$ of the end acts by rotations of the  coordinates $(x^1, \ldots, x^4)$ in a manner that preserves both the background-model Euclidean metric
$\delta$ and  the given K\"ahler metric $g$. 

Because $g$ is K\"ahler by assumption, the associated complex structure $J$ satisfies $\nabla J=0$, where $\nabla$ is the Levi-Civita connection of $g$.  However, 
since our fall-off hypothesis implies that $\nabla = \triangledown + O (|x|^{-2-\varepsilon})$, where $\triangledown$ is the flat Levi-Civita connection of $\delta$,  the  elementary argument presented in 
 \cite[\S 2]{heinleb} shows there is  a  $\delta$-compatible  constant-coefficient almost-complex structure $J_0$
on $\RR^4$   such that 
$$J = J_0 + O (|x|^{-1-\varepsilon}), \qquad \triangledown J = 
O (|x|^{-2-\varepsilon}).
$$
After rotating our coordinates $(x^1, \ldots, x^4)$ if necessary, we may moreover arrange for $J_0$ to to become the standard complex structure 
$$
dx^1\otimes \frac{\partial}{\partial x^2}- dx^2\otimes \frac{\partial}{\partial x^1}+dx^3\otimes \frac{\partial}{\partial x^4}-dx^4\otimes \frac{\partial}{\partial x^3}$$
on $\CC^2$. 
Since the action of the fundamental group $\varGamma_i$ preserves both $J$ and $\delta$, it now follows that 
$\varGamma_i \subset \mathbf{U}(2)$. More importantly, we therefore  automatically  obtain  fall-off conditions 
\begin{equation}
\label{symphony}
\omega  = \omega_0 + O (|x|^{-1-\varepsilon}), \qquad \triangledown \omega = 
O (|x|^{-2-\varepsilon}), 
\end{equation}
for the K\"ahler form $\omega= g(J\cdot, \cdot)$ of $g$, where $\omega_0= dx^1\wedge dx^2+ dx^3\wedge dx^4$ is the standard
symplectic form on $\RR^4 = \CC^2$. 

\begin{prop}
\label{symple}
 Let $(M^4, g, J)$ be an ALE K\"ahler surface, let $M_{\infty,i}$ be an  end of $M$,  let $\widetilde{M}_{\infty,i}$ be the universal cover of $M_{\infty,i}$, 
and let $$(x^1,\ldots , x^4): \widetilde{M}_{\infty,i}\to \RR^4 -B$$ be a diffeomorphism, where $B\subset  \RR^4$ is a standard closed ball of some large radius centered  at the origin. 
 Suppose, moreover,  that these asymptotic coordinates have  been chosen  in accordance with the above discussion, so that 
 the K\"ahler form $\omega$ of $(M,g, J)$  is $C^2$ and  satisfies the fall-off conditions \eqref{symphony}  in this coordinate system, while the 
action of $\pi_1 (M_{\infty,i})$ on $\widetilde{M}_{\infty,i}$ by 
deck transformations is represented in these coordinates by   the action of a finite group $\varGamma_i\subset \mathbf{U}(2)$ of unitary transformations, acting on $\RR^4=\CC^2$
 in the usual 
way. Then there is $\varGamma_i$-equivariant $C^2$-diffeomorphism 
$\Phi: \RR^4- {C} \to \RR^4 - {D}$, where ${C}\subset \RR^4$ is 
a  standard closed ball  centered at the origin, where ${D}\subset \RR^4$ is  a  smooth $4$-ball whose  boundary  $\partial D$ is a  $\varGamma_i$-invariant 
differentiable $S^3$, and where $B\subset C\cap D$, 
 such that    
$$\Phi^* \omega = \omega_0,$$ 
with  $|\Phi (x)-x| = O (|x|^{-\varepsilon})$ and $|\Phi_*-I|= (|x|^{-1-\varepsilon})$.
\end{prop}
\begin{proof}
The following proof is largely  a quantitative  refinement of Moser's stability argument \cite{moser}. 

Let $\mathbf{a}$   denote the radius of the given closed ball  $B\subset \RR^4$, and  notice  that 
we can identify $\RR^4-B$ with $S^3 \times (\mathbf{a}, \infty)$ by means of the  smooth diffeomorphism $x \mapsto (x/|x| , |x|)$. 
Letting  $\varrho =|x| \in (\mathbf{a}, \infty)$ denote the radial coordinate, and letting $\eta =  \frac{\partial}{\partial \varrho}$ denote the 
unit radial vector field in $\RR^4$, we now  define  an ${\zap r}$-dependent  $1$-form $\varphi_{\zap r}$  on $S^3$
by restricting the $1$-form $\eta  \lrcorner \, d\left( \omega - \omega_0 \right)$, which in any case has vanishing radial component,  
to  $S^3\times \{ {\zap r}\} \subset S^3\times (\mathbf{a},\infty )$:
$$\varphi_{\zap r} :=\eta  \lrcorner \, \left( \omega - \omega_0 \right)  \Big|_{\varrho = {\zap r}}, \quad {\zap r}\in (\mathbf{a}, \infty).$$
Because our fall-off conditions guarantee  that $\varphi_{\zap r} = O(r^{-\varepsilon})$ as a $1$-form on $S^3$, 
it follows that, for any  choice of 
$\varrho_0\in (\mathbf{a}, \infty)$,  
$$\psi =  \int_{\varrho_0}^\varrho \varphi_{\zap r}\, d{\zap r} $$
is a  well-defined $\varrho$-dependent $1$-form on $S^3$ of growth $O(\varrho^{1-\varepsilon})$, with  first partial derivatives
 on $S^3$ of similar growth. Viewing $\psi$ as  a $1$-form on  $S^3\times (\mathbf{a}, \infty )$ 
with vanishing component in the $\varrho$-direction,   our assumptions thus not only 
 guarantee it is is a $1$-form of class $C^2$, but also  that its components  
 in $\RR^4$ satisfy  the fall-off conditions 
$$ \psi_k = O(|x|^{-\varepsilon}), \qquad   \psi_{k,\ell} = O(|x|^{-1-\varepsilon}).$$
However, Cartan's magic formula for the Lie derivative tells us that 
$$\mathscr{L}_{\frac{\partial}{\partial \varrho}}\left( \omega - \omega_0 \right) = \eta  \lrcorner \, d\left( \omega - \omega_0 \right) +d
\left[\eta  \lrcorner \left( \omega - \omega_0 \right)\right]= d
\left[\eta  \lrcorner \left( \omega - \omega_0 \right)\right]$$
because $\omega$ and $\omega_0$ are both closed; and since the Lie derivative commutes with $d$ on $C^2$ forms, 
we also have 
$$\mathscr{L}_{\frac{\partial}{\partial \varrho}}d \psi = d [\mathscr{L}_{\frac{\partial}{\partial \varrho}} \psi ] = d\varphi = d
\left[\eta  \lrcorner \left( \omega - \omega_0 \right)\right],$$
too. 
It follows that $\alpha :=\left( \omega - \omega_0 \right)-d\psi$ is a closed, $\varrho$-independent  $2$-form on $S^3 \times (\mathbf{a}, \infty)$.
Moreover, since 
$$\eta \lrcorner \, \alpha = 
\eta \lrcorner \, \left[ \left( \omega - \omega_0 \right) -d\psi\right]= \varphi - \mathscr{L}_\eta \psi = \varphi -\varphi =0,$$
it follows that $\alpha$ is actually the pull-back of a closed $2$-form on $S^3$. But since 
$H^2 (S^3 )=0$, we  therefore have $\alpha = d\beta$ for 
for some $\varrho$-independent $1$-form $\beta$ on $S^3$. 
Moreover, since
$\omega$ and $\omega_0$ are both $\varGamma_i$-invariant, it follows that $\varphi$, $\psi$, and $\alpha$ are all 
$\varGamma_i$-invariant, too;  by averaging, we can therefore arrange for $\beta$ to also be $\varGamma_i$-invariant, while  
still satisfying the equation $\alpha = d\beta$. Setting
 $\theta := \psi + \beta$, we then have    
 $$\omega - \omega_0  = d\theta$$ 
  for a  $\varGamma_i$-invariant $1$-form $\theta$ of class $C^2$  on $\RR^4 - B$ with fall-off
$$\theta = O(|x|^{-\varepsilon}), \qquad   \triangledown \theta = O(|x|^{-1-\varepsilon}).$$

Let us next consider the family of convex combinations
\begin{equation}
\label{combo}
\omega_t = (1-t) \omega_0+ t\omega = \omega_0 + t \left( \omega - \omega_0\right) , \qquad t\in [0,1],
\end{equation}
of the given symplectic  forms $\omega$ and $\omega_0$.  Because $\left( \omega - \omega_0\right)= O (|x|^{-1-\varepsilon})$
as a $2$-form on $\RR^4 -B$,  there is some $\mathbf{b} > \mathbf{a}$ such that 
$|\omega- \omega_0|< 1/\sqrt{2}$ for all $\varrho \geq \mathbf{b}$, where the norm of a $2$-form here  is calculated with respect to 
the Euclidean   metric $\delta$. This then implies that, for any vector $v\in T_x\RR^4= \RR^4$, one  has 
$$|v\lrcorner \, \omega_t| > \frac{1}{2} |v| \qquad \forall t\in [0,1]$$
whenever $\varrho = |x| > \mathbf{b}$; here the vector norm is again measured with respect to the Euclidean metric $\delta$. 
Thus, when $\varrho > \mathbf{b}$ and $t\in [0,1]$, 
 the  maps $T_x\RR^4 \to T^*_x\RR^4$ defined  by the contractions $v\mapsto v\lrcorner\, \omega_t$ are not only invertible, but have
 inverses of operator norm $< 2$ with respect to $\delta$. Defining a $t$-dependent $C^2$ vector field $X_t$ on the exterior region 
 $\varrho \geq \mathbf{b}$ by 
 \begin{equation}
\label{dumb}
X_t  \lrcorner\, \omega_t = -\theta, \qquad t\in [0,1],
\end{equation}
 our fall-off conditions therefore tell us that $|X_t|_\delta= O(\varrho^{-\varepsilon})$ and  $|\triangledown X_t|_\delta = O(\varrho^{-1-\varepsilon})$.
 In particular, it follows that there is some $\mathbf{c} \geq \mathbf{b} + 1$
 such that $|X_t|_\delta < 1$ on the entire region $\varrho \geq  \mathbf{c} -1$, for every $t\in [0,1]$. 
 Also notice that we automatically have
 \begin{equation}
\label{dumber}
 X_t \lrcorner \,\theta = -\omega_t (X_t, X_t) =0,
\end{equation}
 and  that $X_t$ is 
 $\varGamma_i$-invariant, for every $t\in [0,1]$. 
 
 Fixing coordinates $(x^1, \ldots, x^4, t)$ on $\RR^5= \RR^4 \times \RR$, we now consider the closed $2$-form 
 $$\Omega = \omega_0 + d(t\theta)$$
on an open neighborhood of the region $|x| \geq  \mathbf{c} -1$, $0\leq t \leq 1$, 
  where the $t$-independent forms $\omega_0$ and
 $\theta$ are understood to denote the pull-backs of  the corresponding forms on 
$\RR^4$. Since $d\theta = \omega - \omega_0$, we may rewrite this as 
\begin{equation}
\label{jumbo}
\Omega = \omega_t + dt\wedge \theta,
\end{equation}
so that restriction of $\Omega$ to the various  $t=$ constant slices simply yields  the $2$-forms $\omega_t$
of \eqref{combo}. The $C^2$ vector field
$$\xi = \frac{\partial}{\partial t} + X_t$$
on our region of $\RR^5$ therefore satisfies 
$$\xi \lrcorner\, \Omega = \left[ 
\frac{\partial}{\partial t} + X_t
\right] \lrcorner\, 
\left[ 
\omega_t + dt\wedge \theta 
\right] = \theta - \theta =0$$
by dint of  \eqref{dumb}, \eqref{dumber}, and \eqref{jumbo}. Thus 
Cartan's magic formula  now yields 
$$\mathscr{L}_\xi \Omega = \xi \lrcorner \, d \Omega + d \left[ \xi \lrcorner \, \Omega\right] =0.$$
The flow of $\xi$, which simply acts on $t$ by ``time translation,''  therefore locally moves the $2$-form $\omega_t$ on any given time slice to 
the corresponding  $2$-form at a later time. However, because the vector field $X_t$ always has Euclidean length  $|X_t|< 1$ in the region 
$\varrho > \mathbf{c}-1$, the flow-line  of $\xi$ starting at any $(x,0)$ with $|x|\geq \mathbf{c}$ is well-defined for all $t\in [0,1]$, and remains within
$B_1(x) \times [0,1]$. Thus, letting $C$ denote the Euclidean ball $\varrho \leq \mathbf{c}$, there is a family 
$$\Phi_t : \RR^4 - \mathring{C} \to \RR^4, \qquad  t\in [0,1],$$  
of $C^2$ maps given by following the flow of $\xi$ from  $(\RR^4 -\mathring{C})  \times \{ 0\}$ to $\RR^4\times \{ t\}$. These maps are $C^2$ diffeomorphisms
between  $\RR^4 - {C}$ and their images, and satisfy $\Phi_t^* \omega_t= \omega_0$. In particular, $\Phi:= \Phi_1$ provides a symplectomorphism 
between $(\RR^4-C,\omega_0)$ and $({\mathcal U}, \omega)$, for some open set ${\mathcal U}\subset \RR^4$. 
 But since a time-reversed version of our  argument  shows  that backward trajectories of the flow from $\varrho \geq{c}+1$ are also defined 
and  remain 
  in the region $\varrho \geq{c}$ for $t\in [0,1]$, every point in  the region 
$\varrho \geq \mathbf{c}+1$ must belong to the image  $\mathcal{U}$ of $\Phi$. Moreover, 
we can now extend $\Phi$ as a $\varGamma_i$-equivariant
 $C^2$ diffeomorphism $\RR^4\to \RR^4$ by extending the  vector fields $X_t$ to $\RR^4$ while keeping $|X_t|< 1$  by multiplying the fields
defined by \eqref{dumber} by a cut-off function $\phi (\varrho)$ which is $\equiv 1$ for $\varrho > \mathbf{c} -1$ and $\equiv 0$ for $\varrho < \mathbf{c} -1-\epsilon$.
In particular the closed set  $D= \RR^4 -\mathcal{U}$ is actually diffeomorphic to a standard $4$-ball, and its boundary is a $\varGamma_i$-invariant
differentiable $S^3$. 
 Finally, because 
$|X_t|= O(\varrho^{-\varepsilon})$   and  $|\triangledown X_t|_\delta = O(\varrho^{-1-\varepsilon})$. we   have  $|\Phi (x)-x|=O(|x|^{-\varepsilon})$
and $|\Phi_*-I|= (|x|^{-1-\varepsilon})$.
\end{proof} 

\section{Some Useful Symplectic Orbifolds} 

\label{capcom}

If $\varGamma\subset \mathbf{U}(2)$ is a  finite subgroup, the standard action of $\varGamma$ on $\CC^2$  extends to 
$\CP_2 = \CC^2 \sqcup \CP_1$  in an obvious way --- namely,  by letting a $2\times 2$ complex matrix $A$ act on $\CC^3=\CC^2\oplus \CC$ by $A\oplus 1$, and then 
remembering  that $\CP_2 = (\CC^3- 0)/\CC^\times$. Since this construction gives us  an inclusion $\mathbf{U}(2)\hookrightarrow \mathbf{PSU}(3)$, the induced 
action of $\varGamma$ on $\CP_2$  preserves the standard Fubini-Study metric; and since the action also preserves the complex structure 
of $\CP_2$, it also preserves the Fubini-Study K\"ahler form $\omega$,   which we will choose to regard as  a  symplectic form on $\CP_2$. 
 We may therefore   choose to view the quotient 
$(\CP_2, \omega) /\varGamma$ as a symplectic orbifold. 

We  will  henceforth confine our discussion to those $\varGamma$ that {\em act freely on the unit sphere $S^3\subset \CC^2$}. 
Our goal here  will then be to   construct preferred partial desingularizations of every symplectic orbifold  $(\CP_2, \omega) /\varGamma$ that arises in this way.
Of course, if $\varGamma = \{ 1 \}$, then $\CP_2/\varGamma$ is  smooth, so 
there is nothing to do in this regard. We may therefore assume from now on that $\varGamma \neq \{ 1 \}$.
With this assumption,   the origin in $\CC^2\subset \CP_2$ automatically projects  to  a singular point 
$p\in \CP_2/\varGamma$; and since we have assumed that that  $\varGamma$    acts freely on the unit sphere $S^3$, and hence on all of 
$\CC^2-\{ 0\}$, the singular point $p$ is automatically isolated. More specifically, every other singular point  arises from some element of
the ``line at infinity'' $\CP_1\subset \CP_2$.
Our  objective in this section  will  be to  symplectically modify $(\CP_2, \omega) /\varGamma$  in a manner
that leaves the singularity at $p$ unaltered, but eliminates  all  the other singularities. 

In preparation for this, let us first notice that the center $\mathbf{Z}\cong \mathbf{U}(1)$ of $\mathbf{U}(2)$ consists of
scalar multiples of the diagonal matrix, and acts trivially on the $\CP_1$ at infinity.  Moreover,  since $\mathbf{Z}=\mathbf{U}(1)\cong \RR/\ZZ$, the finite group $\mathbf{Z}\cap \varGamma$ must be  cyclic, and thus isomorphic to 
 $\mathbb{Z}_\ell$ for some positive integer $\ell$. Our first step is therefore to consider the quotient $\CP_2/\ZZ_\ell$. Away from the 
 base-point $\hat{p}$ arising from $[0 : 0 :1]\in \CP_2$, this space is topologically non-singular, and can be given a smooth structure such that  $\omega$ 
 descends to it as a symplectic form. This is perhaps  most easily seen via Lerman's theory of symplectic cuts \cite{lerman}; namely, the Fubini-Study symplectic form 
 on $\CP_2$ is obtained  by taking the symplectic  cut at ${\zap H}\leq 1/2$ of $(\CC^2, \omega_0)$ for the Hamiltonian ${\zap H}= (|z_1|^2+|z_2|^2)/2$, which generates
 a free periodic action of period $2\pi$ at and near the boundary. It follows that $\CP_2/\ZZ_\ell$ is simply obtained from $\CC^2/\ZZ_\ell$  by taking the symplectic
 cut at $\widehat{{\zap H}}\leq \ell/2$ for the Hamiltonian $\widehat{{\zap H}} = \ell {\zap H}$, which again generates a free periodic action of period $2\pi$ at and near
 the boundary.\footnote{More generally,  symplectic orbifold singularities of codimension $2$  are always  symplectically invisible. The essential points are that 
 the fixed-point set is automatically  a symplectic submanifold, and  that the area form on  $\CC/\ZZ_\ell$ induced by the 
 standard area form on $\CC$   becomes a constant times 
 the standard area form on $\CC$ if one declares that the complex variable $\zeta= z^{\ell}/|z|^{\ell -1}$ provides an admissible 
chart on the quotient.} If $\ell > 1$, the global quotient  $(\CP_2 , \omega )/\ZZ_\ell$  can thus   be viewed as a  symplectic orbifold $(X_\ell , \omega )$ with exactly one 
single singular point 
$\hat{p}$, corresponding to the origin in $\CC^2$. The symplectic cut construction gives us a symplectic $2$-sphere 
$\Sigma \subset X_\ell$ of self-intersection $+\ell$ that corresponds to the line at infinity $\CP_1\subset \CP_2$,  
and we note in passing   that  $X_\ell -\{ \hat{p} \}$ is actually  diffeomorphic to the $\mathcal{O}(\ell)$  line bundle over $\CP_1$. Since the symplectic condition on 
a submanifold is open, we can also obviously   perturb this
$(+\ell)$-sphere ``at infinity''  so as  to produce an embedded $2$-sphere $\Sigma^\prime$ that  meets $\Sigma$ in only one point, at which $\Sigma$ and $\Sigma^\prime$
are tangent to order $\ell-1$. Moreover, one can do this in such a manner that $\Sigma\cap \Sigma^\prime$ is any chosen point of $\Sigma$, and 
so that   $\Sigma^\prime$ avoids any given  small neighborhood of 
 the singular base-point $\hat{p}$. Indeed, one can even do this explicitly in the present  context, by just taking $\Sigma^\prime$ to be   the image in $\CP_2/\ZZ_\ell$ of a 
 generic complex line in $\CP_2$.
Of course, almost everything said here is also trivially true in the case of $\ell =1$; the only thing that is substantially different about the case of $X_1=\CP_2$ is
that $\hat{p}$ is a non-singular point  when $\ell = 1$.   Whatever the value of $\ell$,  we also automatically have 
\begin{equation}
\label{degree}
\langle c_1 (X_\ell ) , [\Sigma^\prime ]\rangle = \langle c_1 (X_\ell ) , [\Sigma ]\rangle = \chi (\Sigma ) + \Sigma \cdot \Sigma = 2+ \ell \geq 3
\end{equation}
as an immediate consequence of  the adjunction formula. 
 
We now wish to treat the general $\varGamma \subset \mathbf{U}(2)$ that acts freely on $S^3$. We do so by first  noticing that $\CP_2/\varGamma = X_\ell /\check{\varGamma}$, 
where $\check{\varGamma}:= \varGamma/( \mathbf{Z}\cap \varGamma ) = \varGamma/\ZZ_\ell$. Of course, if $\check{\varGamma} = \{ 1\}$, we are already done. 
Otherwise, notice that since  $\varGamma$ acts freely on $S^3$, and hence on $\CC^2-\{ 0\}$, the fact that  $\ZZ_\ell\subset \varGamma$ is central implies that   $\check{\varGamma}$ also acts  freely on $(\CC^2-\{ 0\} ) /\ZZ_\ell$, and hence on $X_\ell - ( \Sigma \sqcup \{ \hat{p}\})$. The singular points of 
$(X_\ell -\{\hat{p}\} /\check{\varGamma}$ therefore all arise from points of $\Sigma \approx S^2$ that are fixed by some non-trivial subgroup of $\check{\varGamma}$. 
However, since $\mathbf{U}(2)/\mathbf{Z} = \mathbf{PSU}(2)\cong \mathbf{SO}(3)$, our group $\check{\varGamma}$  can be thought of as  a finite subgroup of 
$\mathbf{SO}(3)$,  in a way that simultaneously  realizes  the given action of  $\check{\varGamma}$  on $\Sigma$ as the tautological action of 
$\check{\varGamma}\subset \mathbf{SO}(3)$ on $S^2= \mathbf{SO}(3)/\mathbf{SO}(2)$. 
But since the isotropy group $ \subset \mathbf{SO}(3)$ of any point in $S^2$ is  isomorphic to $\mathbf{SO}(2)\cong \RR/\ZZ $, the stabilizer  $\subset \check{\varGamma}$ of any point 
of $\Sigma$ is necessarily  cyclic --- and  of course  is actually trivial for all but a finite number of points! 
 
%
%
%
%
%
%
 
While the above arguments in principle provide all the information we will  need to prove the main result in this section,  it is still worth mentioning the 
 classical fact that  the only possible
finite groups  $\check{\varGamma}\subset \mathbf{SO}(3)$ are the oriented isometry  groups of  a  polygon or regular polyhedron in $\RR^2$ or $\RR^3$, 
and that  this  implies
  that  the 
quotient $\Sigma/\check{\varGamma}$ is  always   a topological $2$-sphere with 
exactly two or three singular points \cite[Chapter 13]{thurston-orb}, which actually arise
 from the orbits of the vertices, edge-centers, and/or face-centers of the corresponding geometric figure.
Here is a  list of the non-trivial possibilities:

\begin{center}
  \begin{tabular}{| l | c | r |}
    \hline
      Group & Figure & Singularities \\ \hline \hline
    \raisebox{.05in}{Cyclic} &\raisebox{-.02in}{\includegraphics[height=.25in]{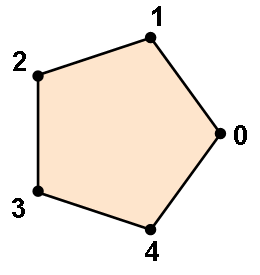}} & \raisebox{.05in}{$\ZZ_n$, $\ZZ_n$} \\ \hline
    \raisebox{.07in}{Dihederal} & \includegraphics[height=.3in]{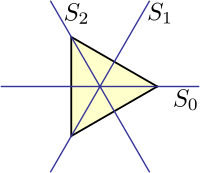} & \raisebox{.07in}{$\ZZ_2$, $\ZZ_2$, $\ZZ_n$} \\ \hline
    \raisebox{.07in}{Tetrahedral} & \includegraphics[height=.3in]{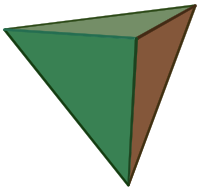} & \raisebox{.07in}{$\ZZ_2$, $\ZZ_3$, $\ZZ_3$}  \\ \hline
     \raisebox{.07in}{Octahedral}& \includegraphics[height=.3in]{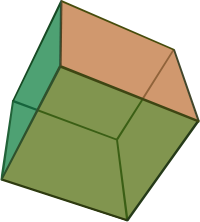} &  \raisebox{.07in}{$\ZZ_2$, $\ZZ_3$, $\ZZ_4$}  \\ \hline
  \raisebox{.07in}{Icosahedral}& \includegraphics[height=.3in]{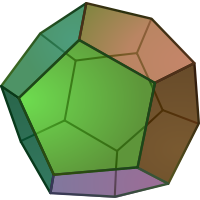} & \raisebox{.07in}{$\ZZ_2$,  $\ZZ_5$,  $\ZZ_5$}  \\
    \hline
  \end{tabular}
\end{center}
Thus, the orbifold  $X_\ell/\check{\varGamma}$ will  actually have exactly $0$, $2$, or $3$ singularities other than the singular   base-point $p$  that is the image of  
$\hat{p}\in X_\ell$.

While this classification does tell us the possible orders of the cyclic groups associated with each singularity, it does not actually 
completely describe the local action of these stabilizers $\ZZ_q\subset \check{\varGamma}$
 on  $X_\ell$, since  the action of $\ZZ_q$  on $\Sigma$ does not determine its action
 on the normal bundle of $\Sigma$.  However, we do know that these stabilizer groups have only isolated fixed points, 
since $\check{\varGamma}$ acts freely on $X_\ell- (\Sigma \sqcup \{ \hat{p}\})$. 
Thus, if $\ZZ_q \subset \check{\varGamma}$ is the stabilizer of some fixed point, its  action is locally
modeled on the action of $\ZZ_q$ on $\CC^2$ generated by 
\begin{equation}
\label{lens}
(z_1, z_2) \mapsto (e^{2\pi i/q}z_1, e^{2\pi i p/q}z_2)
\end{equation}
for a unique integer $p$  with $0< p < q$ and $\mathbf{gcd}(p,q)=1$. 
In complex geometry, there is a standard  minimal resolution of any such singularity, 
obtained by replacing the singular point with a Hirzebruch-Jung string \cite{bpv,hirzebruchjung}, meaning  a finite string 

%
%

\begin{center}
\mbox{
\beginpicture
\setplotarea x from 0 to 260, y from -20 to 80
\put {$-\mathfrak{e}_1$} [B1] at 33 35
\put {$-\mathfrak{e}_2$} [B1] at 70 25
\put {$-\mathfrak{e}_3$} [B1] at 113 35
\put {$-\mathfrak{e}_k$} [B1] at 233 35
\put {$\ldots$} [B1] at 180 30 
{\setlinear
\plot   0 0  80 60  /
\plot   40 60  120 0   /
\plot   80 0  160 60  /
\plot   120 60  160 30   /
\plot   200 30  240  0   /
\plot   200 0  280 60   /
}
\endpicture
}
\end{center}

\noindent 
of copies of $\CP_1$ whose self-intersection numbers $-\mathfrak{e}_j$  are the negatives of the  integers $\mathfrak{e}_j\geq 2$ 
 determined inductively by the algorithm
$$\mathfrak{d}_1= \frac{q}{p}, \quad \mathfrak{e}_j = \left\lceil \mathfrak{d}_{j}\right\rceil , \quad \mathfrak{d}_{j+1}= \left( \mathfrak{e}_j- \mathfrak{d}_{j} \right)^{-1},$$
where the process terminates  at the first $j$ for which $\mathfrak{d}_{j}$ is an integer. 
%
%
%

While it  should be  feasible to carry out  a  symplectic version  of Hirzebruch's construction via a sequence of
symplectic cuts, we will 
instead remove these cyclic singularities by  exploiting \S \ref{logos}. Indeed, for each action \eqref{lens}, Calderbank and Singer \cite{caldsing} 
have constructed a family of ALE scalar-flat Kahler surfaces whose single end is diffeomorphic to  $L(q,p)\times \RR^+$, where $L(q,p) =S^3/\ZZ_q$ is the 
lens space associated with the given action \eqref{lens}. 
The Calderbank-Singer manifolds are, by construction,  diffeomorphic to   Hirzebruch's  minimal resolutions  of $\CC^2/\ZZ_q$, and satisfy 
Chru\'{s}ciel's  fall-off hypotheses with $\varepsilon =1$. For any chosen metric in the family, 
Proposition \ref{symple} thus tells us that the Calderbank-Singer  manifold   contains a compact set whose complement is symplectomorphic to $(\CC^2-\overline{\mathscr{B}_\rad}, \omega_0)/\ZZ_q$,
for the specified   action \eqref{lens} on $\CC^2$, where $\overline{\mathscr{B}_\rad},\subset \CC^2$ is 
is the   standard closed ball  of some radius $\rad> 0$ centered at the origin. By multiplying the Calderbank-Singer metric
(and therefore its K\"ahler form) by a sufficiently small positive constant, we may then arrange
for this statement to actually hold with $\rad$ replaced by 
 any small radius ${\zap r} > 0$ we like. However, each orbifold singularity $y$ we wish to eliminate has a neighborhood modeled on $(\mathscr{B}_{\zap R} , \omega_0)/\ZZ_q$ for some 
radius ${\zap R} > 0$, and for some $\ZZ_q$ action of type \eqref{lens}. 
Choosing our rescaling of the  Calderbank-Singer manifold so that ${\zap r} < {\zap R}$ then allows us to delete a closed neighborhood $(\overline{\mathscr{B}_{\zap r}}, \omega_0)/\ZZ_q$ 
of the singular point $y$, remove the end $(\CC^2-{\mathscr{B}_{\zap R}}, \omega_0)/\ZZ_q$ from the Calderbank-Singer manifold, and 
then glue the  two resulting open manifolds together by identifying the 
two constructed   
copies of the annulus quotient $({\mathscr{B}_{\zap R}}- \overline{\mathscr{B}_{\zap r}},
\omega_0)/\ZZ_q$ via the tautological symplectomorphism between them. 
Since we only need eliminate a finite number of singular points this way, we may also 
take the  radius ${\zap R}$ of these surgery regions  to all be small enough so that  these surgeries
take place in disjoint regions, and so do  not interfere
with each other. Similarly, after choosing some non-singular reference point  $z\in\Sigma/\check{\varGamma}\subset X_\ell/\check{\varGamma}$, 
we also require that the surgery radius ${\zap R}$  be small enough that neither  $z$ nor the singular base-point $p = [\hat{p}]$ belongs to the
closure of 
these surgery regions.  

We will now verify that this construction  proves  the following result:

\begin{prop} 
\label{encapsulate}
Let $\varGamma\subset \mathbf{U}(2)$ be a finite subgroup $\neq \{1\}$ that  acts freely on the unit sphere $S^3\subset \CC^2$. 
Then there is a $4$-dimensional compact connected symplectic orbifold $(X_\varGamma, \omega_\varGamma )$ such that 
\begin{enumerate}[(I)]
\item $(X_\varGamma, \omega_\varGamma )$ contains exactly one singular point $p$; \label{un}
\item $p$ has a neighborhood symplectomorphic to $(\mathscr{B},\omega_0)/\varGamma$ for some standard open ball $\mathscr{B}\subset \CC^2$
centered at the origin, where $\varGamma$ acts on $(\CC^2,\omega_0)$ in the tautological  manner, as a subgroup of $\mathbf{U}(2)$; and \label{deux}
\item there is a symplectic immersion  ${\zap j}: S^2 \looparrowright X_\varGamma-\{p\}$, with at worst transverse positively-oriented   double points, 
 such that $$\int_{S^2} {\zap j}^*[c_1(X_\varGamma-\{p\},J)] \geq 3$$
for some, and hence any,  $\omega$-compatible almost-complex structure $J$.\label{trois} 
\end{enumerate}
\end{prop}
\begin{proof} We need only check the last condition, since the first two  are obviously satisfied as long as  the surgery radius ${\zap R}$ is 
small. To produce the immersed sphere promised by condition \eqref{trois}, recall that we can construct an embedded sphere $\Sigma^\prime\subset 
X_\ell-\hat{p}$ of self-intersection $\ell$ that only touches $\Sigma$ at a chosen point of the latter $2$-sphere. 
Let us now take the point $\Sigma\cap \Sigma^\prime$ to be one whose  stabilizer under the action of $\check{\varGamma}$ is trivial, 
so that it projects to a non-singular point  $z\in\Sigma/\check{\varGamma}\subset X_\ell /\check{\varGamma}$.
By shrinking the surgery radius  ${\zap R}$,  we can then guarantee that $z$ lies outside the closure of the surgery regions. Also recall
that one can take $\Sigma^\prime$ to be the image in $X_\ell= \CP_2/\ZZ_\ell$ of a   projective line $\CP_1\subset \CP_2$ that avoids 
 the origin $[0:0:1]$, is not the line at infinity, and passes 
through the point of $\CP_2$ that maps to   $z$. This construction 
in particular guarantees that $\Sigma^\prime- \{z\}$ is a holomorphic curve with respect to a complex structure on $X_\ell - (\Sigma\sqcup \{ \hat{p}\})$
that is invariant under the action on ${\varGamma}\subset \mathbf{U}(2)$. Projecting $\Sigma^\prime$ to  $X_\ell/\check{\varGamma}$ 
therefore gives us an immersed symplectic $2$-sphere whose self-intersections all 
belong to the open set $\mathcal{V}:=[X_\ell - (\Sigma\sqcup \{ \hat{p}\}]/\check{\varGamma}$. It follows that these self-intersections
  are  all transverse and positive, because in this region  our sphere is a totally geodesic holomorphic curve with respect to 
  the metric and $\omega$-compatible complex structure induced of $\mathcal{V}$ by the Fubini-Study metric and complex structure
  on $\CP_2$.  If necessary, we can then smoothly 
perturb this immerse $2$-sphere near any triple points in order to produce a    symplectic  immersion ${\zap j}: S^2 \looparrowright X_\varGamma-\{p\}$ that has 
at worst
transverse, positively-oriented double points. Since the image of ${\zap j}$ is closed and avoids all the the singular points of $X_\ell/\check{\varGamma}$, we can also arrange that 
is disjoint from the surgery regions by shrinking the surgery radius ${\zap R}$ if necessary. Finally, since ${\zap j}^*c_1$ coincides with the restriction of $c_1(X_\ell)$ to 
$\Sigma^\prime$, it follows that   $\int_{S^2}{\zap j}^*c_1=2+\ell \geq 3$ by \eqref{degree}. 
\end{proof}

Since the orbifolds we have just constructed will play an essential role in the next section, it now seems  appropriate to give them a name: 

\begin{defn} 
\label{capsule}
Let $\varGamma\subset \mathbf{U}(2)$ be any finite subgroup  that  acts freely on the unit sphere $S^3\subset \CC^2$. 
Then
\begin{itemize}
\item 
If $\varGamma \neq \{1\}$,  a {\em $\varGamma$-capsule} will mean one of the standard symplectic orbifolds $(X_\varGamma, \omega_\varGamma )$
satisfying conditions \eqref{un}--\eqref{trois}
that  we have 
constructed in this section. The unique singular point $p\in X_\varGamma$ will then be called the {\em base-point} of the $\varGamma$-capsule.
\item When $\varGamma = \{1\}$, we instead define the associated $\varGamma$-capsule $(X_\varGamma, \omega_\varGamma )$ to be $\CP_2$, equipped with its
standard Fubini-Study symplectic structure. In this case, the base-point $p$ of $X_\varGamma$ will simply mean $[0:0:1]\in \CP_2$. 
\end{itemize}
\end{defn}

Thus,   Proposition \ref{encapsulate} can be restated as saying that, whenever 
$ \varGamma\subset \mathbf{U}(2)$
is a finite subgroup
that acts freely on $S^3$,  there always exists 
 a $\varGamma$-capsule. 

\section{Capping Off the Ends}

Now suppose that $(M^4,g,J)$ is an ALE K\"ahler manifold of complex dimension $2$, where the metric is merely assumed to  satisfy the Chru\'{s}ciel fall-off conditions 
 \eqref{eins}-\eqref{zwei} for some $\varepsilon > 0$, in some real coordinate system at each end. This definition does not obviously exclude the possibility that 
$M$ might actually have several ends. However, our first main result is that such a scenario is actually impossible: 

\begin{thm} 
\label{tend}
Let $(M^4,g,J)$ is an ALE K\"ahler surface, where the metric is merely assumed to satisfy the fall-off hypotheses  \eqref{eins}-\eqref{zwei} for some $\varepsilon > 0$.
Then $M$ has exactly one end. 
\end{thm}

\begin{proof}  
For each end $M_{\infty,i}\approx (S^3/\varGamma_i)\times \RR^+$ of $M$, we may first choose  some $\varGamma_i$-capsule $(X_\varGamma, \omega_{\varGamma_i} )$, 
the existence of which is guaranteed by Proposition \ref{encapsulate}. By Proposition \ref{symple}, each end $M_{\infty,i}$ contains an asymptotic region 
symplectomorphic to $(\CC^2 - \overline{\mathscr{B}_{\mathfrak{R}}}, \omega_0)/\varGamma_i$ for some sufficiently large common 
radius $\mathfrak{R}$. On the other hand, the base-point of each 
$\varGamma_i$-capsule has a neighborhood symplectomorphic to $(\mathscr{B}_{\mathfrak{r}}, \omega_0)/\varGamma_i$ for some small common radius $\mathfrak{r}$,
and by shrinking this radius if necessary we can guarantee that this ball-quotient in each $\varGamma_i$-capsule 
does not meet some chosen symplectically  immersed $2$-sphere satisfying \eqref{trois}. 
We now inflate the $\varGamma_i$-capsules by replacing their symplectic forms $\omega_{\varGamma_i}$ by $t^2\omega_{\varGamma_i}$ for some large $t> 0$. In the inflated 
$\varGamma_i$-capsules, the base-point now has a neighborhood symplectomorphic to $(\mathscr{B}_{R}, \omega_0)/\varGamma_i$, where 
$R= t\mathfrak{r}$. Thus, by taking $t$ to be sufficiently large, we may arrange that $R>\mathfrak{R}$. By now removing $\overline{B_{\mathfrak{R}}}/\varGamma_i\ni p$
from each $\varGamma_i$-capsule and   $(\CC^2 - \overline{\mathscr{B}_{R}})/\varGamma_i$ from each $M_{\infty,i}$, we are then left with pieces
that may be glued together symplectically along copies of $({\mathscr{B}_R}-\overline{B_{\mathfrak{R}}})/\varGamma_i$ to produce a compact
symplectic $4$-manifold $(N, \hat{\omega})$. 

Now $(N, \hat{\omega})$ has been constructed so that it contains a symplectically immersed $2$-sphere ${\zap j}_i:S^2 \looparrowright N$  in each capped-off end. 
Moreover, this $2$-sphere has at worst positive transverse double points, and satisfies $\int_{S^2}{\zap j}^*c_1  \geq 3$. If the sphere has any double points at all, a result of McDuff
\cite[Theorem 1.4]{mcmori}  then tells us that $N$ symplectomorphic to a rational complex surface, and so  orientedly diffeomorphic to either  $S^2\times S^2$ 
or $\CP_2\# k \overline{\CP}_2$ for some $k\geq 0$. On the other hand, if the sphere has no double points, it is then an embedded symplectic $2$-sphere
of self-intersection $\geq 3-2=1> 0$, so an earlier result of McDuff  \cite[Corollary 1.6]{mcrules} once again tells us that  $N$ is orientedly diffeomorphic to a rational complex surface.
In particular, it follows that $b_+(M)=1$, meaning that the intersection form $H^2(M, \RR) \times H^2 (M, \RR)\to \RR$ is of type $({+}{-}\cdots{-})$. 

Now each of the immersed spheres ${\zap j}_i(S^2)$ we have constructed can be modified to yield a connected embedded  symplectic surface $\mathscr{S}_i$
by replacing a small neighborhood of each double point with a cylinder  $S^1\times (-\epsilon , \epsilon )$. This process increases the genus, but does not change the 
homology class; moreover, it can be carried out while remaining completely  inside  the truncated $\varGamma_i$-capsule containing ${\zap j}_i(S^2)$. We therefore
have 
$$\langle c_1(N),  [\mathscr{S}_i]\rangle = \int_{S^2} {\zap j}_i^*c_1 \geq 3,$$
and, since $\mathscr{S}_i$ is symplectic and embdedded,   the adjunction formula allows us to rewrite this as 
$$\chi (\mathscr{S}_i) + [\mathscr{S}_i]\cdot [\mathscr{S}_i] \geq 3.$$
But since $\mathscr{S}_i$ certainly has Euler characteristic $\chi (\mathscr{S}_i)\leq 2$, it therefore follows that  
$$[\mathscr{S}_i]\cdot [\mathscr{S}_i] \geq  3-\chi (\mathscr{S}_i)   \geq  3-2= 1,$$ 
so  each of these surfaces has positive homological self-intersection. However, notice that  $\mathscr{S}_i\cap \mathscr{S}_j= \varnothing$ when $i\neq j$,
since the truncated $\varGamma$-capsules where they live are, by construction, disjoint. This has the homological consequence that  
$$ [\mathscr{S}_i]\cdot [\mathscr{S}_j] =0 \quad \forall i\neq j.$$
It follows that  $b_+(N)$ is at least as large as  the number of ends of $M$. But since we have also just  seen  that $b_+(N)=1$, 
this means that there can be at most one end.   As our definition of an 
ALE manifold moreover requires $M$ to be non-compact, it therefore follows that M has exactly one end.
\end{proof}

\begin{rmk} The regularity of the gluing maps used in the  above construction depends, via Proposition \ref{symple}, on the regularity of the given metric 
$g$. Thus, if $g$ is merely $C^2$, our symplectic manifold $(N, \hat{\omega})$ is ostensibly merely a symplectic manifold with $C^2$ coordinate transformations
between Darboux coordinate charts. This might lead  one to  worry, because many of the cited papers in symplectic topology implicitly assume
   that all objects under discussion are of class 
$C^\infty$. Fortunately, such fears are misplaced, for  general reasons  we  will now explain. Indeed, by a celebrated result of Whitney \cite{whitney},
there exists  a smooth structure on $N$ which is compatible with the given $C^2$ structure, and  the $C^\infty$  $2$-forms, defined relative to this chosen 
smooth structure, will then  be dense among $C^1$ closed forms 
in  the cohomology class  $[\hat{\omega}]$. However, if the smooth form $\tilde{\omega}\in [\hat{\omega}]$ is 
sufficiently close to $\hat{\omega}$ in the $C^1$ topology, all the convex combinations $(1-t)\hat{\omega}+ t\tilde{\omega}\in [\omega ]$, $t\in [0,1]$,  will   be symplectic forms, 
and Moser's stability argument \cite{moser} will then produce a $C^1$ symplectomorphism between the $C^2$ symplectic manifold 
$(N,\hat{\omega})$ and the
smooth symplectic manifold $(N,\tilde{\omega})$. Thus, our use of classification results for smooth symplectic manifolds  is entirely justified. 
\end{rmk}

\begin{cor}
The mass formula of \cite{heinleb}  holds in all complex dimensions $\geq 2$, merely assuming  Chru\'{s}ciel fall-off conditions 
 on the metric. In particular, if  $(M^4,g,J)$ is an ALE K\"ahler surface that merely satisfies  \eqref{eins}-\eqref{zwei} for some $\varepsilon > 0$, then 
 its mass is given by 
 \begin{equation}
\label{massive}
{\zap m}(M,g) =  -{\frac{1}{3\pi}} \langle \clubsuit (c_1) , [\omega ]\rangle+
{\frac{1}{12\pi^2}}\int_M s_g d\mu_g
\end{equation}
  where $s_g$ and $d\mu_g$ are the scalar curvature and metric volume form,  $c_1$ is the first Chern class of $(M,J)$, 
  $\clubsuit$ is the inverse of the natural homomorphism $H^2_c(M)\to H^2(M)$, and 
  $\langle ~, ~\rangle$ is the natural duality pairing 
  between  $H^2(M)$ and $H^2_c(M)$. 
\end{cor}
\begin{proof} The case of complex dimension $\geq 3$ was already proved in \cite{heinleb}. In the case of complex dimension $2$, 
 the proof of \cite[Theorem 5.1]{heinleb} now proves the claim as long as one replaces the citation   of \cite[Proposition 4.2]{heinleb}  with  a reference to 
 the above Theorem \ref{tend}.
\end{proof}

Because there are other plausible methods available for proving   Theorem \ref{tend}, some might wonder if 
all our work in 
\S \ref{capcom} was worth the effort. Fortunately,  the ideas we have described here have other consequences which provide further justification 
for the current project:

\begin{prop}
If $(M^4,g,J)$ is any ALE K\"ahler surface  with Chru\'{s}ciel metric fall-off, 
then $M$ is diffeomorphic to the complement of a tree of symplectically embedded  $2$-spheres in 
a rational complex surface. 
\end{prop}
\begin{proof}
By a tree of embedded $2$-spheres, we mean a union of transversely intersecting embedded symplectic $2$-spheres such that the dual graph representing their 
intersection patten is connected and contains no loops. 
The tree we have in mind here is determined by $\varGamma$, and is  specifically the subset of a $\varGamma$-capsule gotten by 
attaching the appropriate      Hirzebruch-Jung string to each orbifold point of 
$\Sigma/\check{\varGamma}\approx S^2$. Since  the proof of Theorem \ref{tend} shows  that $M$ can be diffeomorphically compactified  
into a rational symplectic  manifold $N$ by attaching a truncated $\varGamma$-capsule $Y= X_\varGamma- \overline{\mathscr{B}}/\varGamma$, 
the result follows from the fact that the complement of the obvious tree in $Y$ is  diffeomorphic to $(S^3/\varGamma) \times (0,1)$. 
\end{proof}

Here is another immediate consequence of the same ideas:

\begin{prop}
For any ALE K\"ahler surface $(M^4,g,J)$ with Chru\'{s}ciel metric fall-off, 
the fundamental group of $M$ is finite. 
\end{prop}
\begin{proof}
Once again,  the proof of Theorem \ref{tend} shows that  compactifying $M$ by adding a truncated  $\varGamma$-capsule $Y$ results in a  symplectic
$4$-manifold $N$ that is diffeomorphic to a rational complex surface. In particular, this assertion means  that $N$ is simply connected. However,  we also have
$$N= M\cup Y, \qquad M\cap Y  \approx (S^3/\varGamma ) \times (0,1) \approx M_\infty,$$
where   $Y$ is obtained from a $\varGamma$-capsule $X_\varGamma$   by removing a closed neighborhood 
$\overline{\mathscr{B}_{\zap r}}/\varGamma$ 
of the 
base-point $p$. However, since $Y$ deform retracts to  the tree of $2$-spheres obtained by attaching a     Hirzebruch-Jung string to each orbifold singularity of 
$\Sigma/\check{\varGamma}\approx S^2$, it follows that   $Y$ is    simply connected. The Seifert-van Kampen  theorem therefore tells us that $\pi_1(N)$ is the quotient of
$\pi_1(M)$ by the image of $\pi_1 (M\cap Y) \cong \varGamma$. But since  $\pi_1(N)=\{1\}$, this means  that $\Gamma \to \pi_1(M)$ is surjective.
In particular, 
 $\pi_1(M)$ is necessarily  finite. 
\end{proof}

It is worth emphasizing  that, despite persistent rumors  to the contrary, $M$  really might not be simply connected, even in the Ricci-flat case. 
For pertinent  examples and   classification results, see \cite{isuvaina,epwright}.  

\bigskip 

Of course, the simplest case of the present story is when the manifold in question is {\em asymptotically Euclidean} (AE);
these are the special ALE manifolds for which 
 $\varGamma=\{ 1\}$. It is only in this setting that one can hope to prove a positive mass theorem \cite{lp,symass,syaction,witmass},
 asserting that that non-negative scalar curvature necessarily implies non-negative mass; in the more general ALE 
 setting, such statements are typically  false  \cite{lpa}. But in the AE  setting, one can even sometimes  prove Penrose-type
 inequalities \cite{braymass,huilmpen,penineq}, which offer  lower bound for
 the mass in terms of the areas of suitable minimal submanifolds of the space  in question. In the 
 K\"ahler context,  a sharp lower bound of this type was given by \cite[Theorem E]{heinleb}. However, while our proof of this
 result only required Chru\'{s}ciel  fall-off in complex dimensions $\geq 3$, we needed to assume stronger
 fall-off in hypotheses complex dimension $2$. 
 
 Fortunately, the ideas  developed here  provide a way around this difficulty.

\begin{thm}[Penrose Inequality for K\"ahler Manifolds]
\label{epsilon}
Let $(M^{2m},g,J)$ be an AE K\"ahler manifold,, where  the metric merely satisfies the  Chru\'{s}ciel  fall-off hypotheses \eqref{eins}-\eqref{zwei}
for some $\varepsilon > 0$ in some real asymptotic coordinate system. If the scalar curvature $s$ of $g$ is everywhere non-negative, 
then $(M,J)$ carries a numerically  canonical divisor $D$ that is expressed as a sum $\sum n_jD_j$
of compact complex hypersurfaces with  positive integer coefficients, with the property that
 $\bigcup_j D_j \neq \varnothing$  
whenever  $(M,J)$ is not diffeomorphic to $\RR^{2m}$. 
In  terms of this divisor, the mass of the manifold then satisfies
$${\zap m}(M,g) \geq  \frac{(m-1)!}{(2m-1)\pi^{m-1}} \sum_j  n_j \mbox{Vol}\, (D_j) 
$$
and equality holds if and only if $(M,g,J)$  is scalar-flat K\"ahler. 
\end{thm}
\begin{proof}
Since this was already proved in \cite{heinleb} in complex dimension  $\geq 3$, we may henceforth restrict ourselves to the case where  $(M^4,J)$
is a complex surface. In this case, the proof of Theorem \ref{tend} shows that we can produce a compact symplectic manifold $(N,\omega)$ by removing
a standard symplectic end $(\RR^4 - {\mathscr{B}_{\mathfrak{R}}},\omega_0)$ and replacing it  $\CP_2$ minus a ball, equipped with some multiple of the 
Fubini-Study symplectic form. In this setting, a projective line in $\CP_2$ gives us a symplectic $2$-sphere of self-intersection $+1$ in $(N,\omega)$. 
A result of McDuff \cite[Corollary 1.5]{mcrules} then tells us that $(N,\omega)$ is  symplectomorphic to a blow-up of $\CP_2$, equipped with some
K\"ahler form, in a way that sends the given $2$-sphere to a projective line $\CP_1$ that avoids all the blown-up points. Removing this ``line at infinity,'' 
we thus see that 
 $M$ must be diffeomorphic to $\RR^4 \# k\overline{\CP}_2$, where $k=b_2(M)$,  and $H_2(M, \ZZ)$  is moreover generated by the homology 
 classes of $k$ disjoint symplectic
 $2$-spheres $E_1, \ldots , E_2\subset M$ of self-intersection $-1$. But then, under the natural identification $H^2_c(M)=H_2(M,\RR )$ arising from 
 Poincar\'e-Lefschetz  duality, we then have $\clubsuit (-c_1)= \sum_{i=1}^k[E_i]$, as may be checked by integrating/intersecting both sides against each 
 of the homology generators
 $[E_j]$. Thus, the  mass formula \eqref{massive} tells us in the AE case that 
 $${\zap m}(M,g) =  \frac{1}{3\pi}\sum_{j=1}^k \int_{E_j}\omega + \frac{1}{12\pi^2}\int_M s_g d\mu_g .$$

We now show that each of the homology classes $E_i$ can actually be represented by a finite
 sum of holomorphic  curves $D_j$ in $(M,J)$ with positive integer coefficients. We do this by first
 carrying out our construction of the compact symplectic manifold $(N,\hat{\omega})$ rather more carefully.  First, notice  that
 our metric fall-off condition \eqref{zwei} guarantees that the vector field $\eta := (\varrho \nabla\varrho )/|\nabla \varrho|^2$
 defined in term of the Euclidean radius $\varrho$ and the metric  $g$, satisfies ${\mathcal L}_\eta g = 2g + O(\varrho^{ -1- \varepsilon} )$;
 moreover, the generalized Euler vector field $\eta$ is normal to the spheres $\varrho=$ constant, and its flow just rescales the radial function $\varrho$ by positive constants. 
 For $\mathbf{c}$ is sufficiently large, we may therefore 
 define a  distance-nonincreasing piecewise differentiable map $\Psi : M \to M$ that sends the inner region $\varrho \leq \mathbf{c}$ 
 to itself by the identity, and that sends  the outer region $\varrho \geq \mathbf{c}$ to the boundary sphere $\varrho = \mathbf{c}$ by the backward flow of $\eta$. 
 By choosing $\mathbf{c}$ to be sufficiently large, we can also arrange  that the restriction of $\Psi$ to the region $\varrho \geq 3\mathbf{c}$ actually contracts 
 distances by a factor of at least  $2$. 
 
 We next apply the coordinate transformation $\Phi$ given by Proposition \ref{symple} in order to identify the K\"ahler form $\omega$ on the asymptotic region  of 
 $(M,g,J)$ with the standard symplectic form $\omega_0$ on $\CC^2$. Because the derivative of $\Phi$ satisfies $\Phi_* = I + O (|x|^{-1-\varepsilon})$, the image
 $\Phi_*J_0$ is uniformly as close as we like to $J_0$ in the image of the region $\varrho \geq \mathbf{c}$, provided we again  take $\mathbf{c}$ to be 
 sufficiently large. Our fall-off condition $J = J_0 + O (|x|^{-1-\varepsilon})$ now also  guarantees that $\tilde{J}= \Phi_*J$ is similarly  uniformly close to 
 $J_0$. In particular, we may arrange that $T^{1,0}_{\tilde{J}}\cap T^{0,1}_{J_0}=0$, which then allows us to represent 
 $T^{1,0}_{\tilde{J}}$ by a tensor field $\phi\in \Lambda^{0,1}_{J_0}\otimes T^{1,0}_{J_0}$, and the fact that $\tilde{J}$ and $J_0$ are both 
 $\omega_0$ compatible is then encoded by the statement that $\phi \lrcorner\, \omega_0 \in \Lambda^{0,1}_{J_0}\otimes \Lambda^{0,1}_{J_0}$ is symmetric.
 Since the latter condition is linear in $\phi$, the almost-complex structure $\hat{J}$ corresponding to $f\phi$, will also be $\omega_0$ compatible, where we now
 take $f=f(\varrho)$ to be a smooth, non-increasing cut-off function which is $\equiv 1$ for $\varrho\leq 4\mathbf{c}$ and $\equiv 0$ for $\varrho\geq 5\mathbf{c}$.
 Because this almost-complex structure is still uniformly close to $\tilde{J}$, the corresponding Riemannian metric $\hat{g}=\omega_0(\cdot, \hat{J}\cdot)$
 is uniformly close  to $g$ in the exterior region, and we can therefore arrange that $\Phi^*\hat{g} \geq g/2$ in the region $\varrho \geq 3\mathbf{c}$, while
 nonetheless keeping $\Phi^*\hat{g}= g$ in the region $\varrho \leq 3\mathbf{c}$. Thus, the constructed map $\Psi : M \to M$ is distance non-increasing 
 with respect to $\hat{g}$ as well as with respect to $g$; and it is moreover {\em strictly} distance decreasing  outside of  the region 
where $\varrho \leq \mathbf{c}$ in our original coordinates. 
 
To cap off the end, we next choose a K\"ahler metric   $h$  
 on $\CP_2$  that  is identically Euclidean on the unit ball in $\CC^2\subset \CP_2$. By multiplying $h$ by a 
 large positive constant $\lambda > 25 \mathbf{c}^2$, we then obtain a K\"ahler metric $\lambda h$ on $\CP_2$ which contains an isometric copy of a  Euclidean ball of 
 radius $> 5 \mathbf{c}$.  We then cut a Euclidean ball $\mathcal{B}$ of radius $5 \mathbf{c}$ out of this larger ball, and glue in the region $\mathcal{U}\subset M$ that is given by 
 $\varrho \leq 5 \mathbf{c}$ in our symplectic coordinates.  The resulting symplectic $4$-manifold $(N, \hat{\omega})$  thus comes equipped with 
 an almost-K\"ahler metric $\hat{g}$ which is given by $\lambda h$ on $\CP_2-\mathcal{B}$, by $g$ on the region $\mathcal{V}\subset \mathcal{U}$ corresponding
 $\varrho \leq \mathbf{c}$ in our initial coordinates, and by the constructed interpolation $\hat{g}$ on the transition annulus $\mathcal{U}- \mathcal{V}$.

  However, because $(N,\hat{\omega})$ is a symplectic manifold with $b_+=1$,  a result of Taubes 
  \cite{taubes} therefore tells us that  the perturbed Seiberg-Witten invariant  
 of $N$ is  non-zero for the the spin$^c$ structure $\mathfrak{c}$ determined by $J$ and the chamber containing large  multiples of $-[\hat{\omega} ]$. However,  because 
 $N$ also admits self-diffeomorphisms which act on $H^2(N)$ by $[E_i] \mapsto -[E_i]$ and by the identity on $[E_i]^\perp$, the analogous perturbed Seiberg-Witten invariant is 
 also non-zero for the images of $\mathfrak{c}$ under all these reflections. 
 It therefore follows  \cite{liliu2,taubes3} that 
 each of the classes
 $[E_i]$ is represented by a (possibly singular) $\hat{J}$-holomorphic curve $\mathscr{E}_i$. Moreover, $\mathscr{E}_i$ is the zero locus
  of a section $u$ of a line bundle $\mathscr{L}_i\to M$ with Chern class $c_1(\mathscr{L}_i)=[E_i]$ with the property that $u$  is approximately holomorphic near $\mathscr{E}_i$. 
 
 Now the truncated capsule  region $\CP_2 - \mathcal{B}$ of $N$ is a union of projective lines, and these, by construction, are all $\hat{J}$-holomorphic curves. 
 Since $u$ is approximately holomorphic near $\mathscr{E}_i$, the number of zeroes of  the restriction of $u$ to any such projective line $P$, counted with  the obvious 
 non-negative multiplicities,  is exactly $\int_P c_1(\mathscr{L}_i)$. 
However, $\int_P c_1(\mathscr{L}_i)$ is also exactly  the intersection pairing of $[E_i]$ and $[P]$, which we have known from the outset to be  zero. It follows that  $u$ is everywhere non-zero on every  such projective line $P$, 
so that we always have $\mathscr{E}_i\cap P=\varnothing$. But since  $\CP_2 - \mathcal{B}$ is a union of such  projective lines $P$, this implies  that 
 $\mathscr{E}_i\subset N- (\CP_2 - \mathcal{B}) =\mathcal{U}$. 
 
 This means that $\mathscr{E}_i$ is a pseudo-holomorphic curve in $(\mathcal{U},\hat{J})$, and thus of $(M,\hat{J})$, where we we now recall that 
 our interpolated almost-complex structure  $\hat{J}$ was initially defined in symplectic coordinates on the entire end $M_\infty$. Here it is worth
 pointing out that, while $\mathscr{E}_i$ may very well be singular, the corresponding pseudo-holomorphic curves for generic perturbations
 $J^\prime$ of $\hat{J}$ are embedded $2$-spheres because $[E_i]^2 =-1$ and $c_1\cdot [E_i]=+1$; by Gromov compactness  \cite{gromsym,mcsal},
$\mathscr{E}_i$ can therefore be, at worst, a finite tree of branched minimal $2$-spheres. We now recall  that, since  these $2$-spheres are all calibrated 
submanifolds of the almost-K\"ahler manifold $(M,\hat{g}, \omega)$, each one has least area among surfaces its homology class. 
But we have carefully arranged for the piecewise smooth map $\Psi : M\to M$ to be distance non-increasing with respect to $\hat{g}$, and to even 
be  strictly distance decreasing on $M-\mathcal{V}$; moreover, $\Psi : M\to M$ was also constructed as a deformation retraction of $M$ to $\mathcal{V}$.
It therefore follows that none of the $2$-spheres that make up $\mathscr{E}_i$ cannot meet $M-\mathcal{V}$, because applying $\Psi : M\to M$ to 
such a $2$-sphere would otherwise produce a homotopic $2$-sphere of strictly smaller area. It therefore follows that each spherical piece of $\mathscr{E}_i$,
and hence  the entire  pseudo-holomorphic curve
$\mathscr{E}_i$ itself, must be contained in $\mathcal{V}$, where $\hat{J}$ coincides with the original integrable complex structure $J$ of
$(M,J)$. In other words, each $\mathscr{E}_i$ is actually a holomorphic curve in our original K\"ahler manifold $(M,g,J)$. This means that 
$\int_{E_i} \omega$ is in fact exactly the area of $\mathscr{E}_i$, counted with multiplicities, and our mass formula can therefore be rewritten as
$${\zap m}(M,g) =  \frac{1}{3\pi}\sum_{i} \Vol ({\mathscr{E}_i})  + \frac{1}{12\pi^2}\int_M s_g d\mu_g.$$
If the $D_j$ are the various spherical components of the various $\mathscr{E}_i$, and if $n_j$ is the multiplicity with which a given $D_j$ occurs
in this way, can then rewrite this as
$${\zap m}(M,g) =  \frac{1}{3\pi}\sum_{j} n_j \Vol (D_j)  + \frac{1}{12\pi^2}\int_M s_g d\mu_g.$$
If $s_g \geq 0$, this then gives us  the Penrose-type inequality 
$${\zap m}(M,g) \geq  \frac{1}{3\pi}\sum_{j} n_j \Vol (D_j),$$
where equality iff  $g$ is scalar-flat K\"ahler. 
\end{proof}

There is one respect in which this result remains noticeably weaker than \cite[Theorem E]{heinleb}. Indeed, the earlier argument shows 
that, assuming stronger fall-off conditions, the underlying complex surface of an AE $(M,g,J)$ must be  an iterated blow-up of $\CC^2$. 
What we have essentially shown  here is    that the the weaker fall-off conditions \eqref{eins}-\eqref{zwei}  imply that $(M^4,J)$ 
is  an iterated blow-up of a complex surface diffeomorphic to $\RR^4$. Nonetheless, this is quite good enough for applications like the 
following:

\begin{cor}[Positive Mass Theorem for K\"ahler Manifolds] 
Let $(M^{2m},g,J)$ be an AE K\"ahler manifold,, where  the metric merely satisfies the  Chru\'{s}ciel  fall-off hypotheses \eqref{eins}-\eqref{zwei}
for some $\varepsilon > 0$ in some real asymptotic coordinate system. If $g$ has  scalar curvature $s_g \geq 0$ everywhere, 
then ${\zap m}(M,g) \geq 0$, with equality iff $(M,g)$ is Euclidean space. 
\end{cor}
\begin{proof} 
By Theorem \ref{epsilon}, we merely need consider the case when $M$ is diffeomorphic to $\RR^{2m}$ and the metric $g$ is scalar-flat K\"ahler. 
However, this implies that the Ricci-form $\rho$ of  $g$ is harmonic, and is an $L^2$ harmonic  form. Since de Rham classes on an ALE manifold
have unique harmonic representatives, this means that $g$ must be Ricci-flat, because we have assumed that
 $M$ is contractible. But  since the asymptotic volume growth of an AE metric is 
exactly Euclidean, the
Bishop-Gromov equality therefore implies that  the exponential map gives  an isometry between any tangent space 
and  $(M,g)$. 
\end{proof}

\vfill

\noindent
{\bf Acknowledgments.} This paper was written during a stay at the \'Ecole Normale Sup\'erieure in Paris,
while the author was   on sabbatical  leave as a Simons Fellow. 
It is a particular pleasure to  thank Olivier Biquard  for   his gracious hospitality in Paris, as well as
for many stimulating and useful conversations. He would also like to thank the faculty of the ENS for  
offering a warm welcome into their idyllic research environment.


\begin{thebibliography}{10}

\bibitem{bpv}
{\sc W.~Barth, C.~Peters, and A.~Van~de Ven}, {\em Compact Complex Surfaces},
  vol.~4 of Ergebnisse der Mathematik und ihrer Grenzgebiete (3),
  Springer-Verlag, Berlin, 1984.

\bibitem{bartnik}
{\sc R.~Bartnik}, {\em The mass of an asymptotically flat manifold}, Comm. Pure
  Appl. Math., 39 (1986), pp.~661--693.

\bibitem{braymass}
{\sc H.~L. Bray}, {\em Proof of the {R}iemannian {P}enrose inequality using the
  positive mass theorem}, J. Differential Geom., 59 (2001), pp.~177--267.

\bibitem{caldsing}
{\sc D.~M.~J. Calderbank and M.~A. Singer}, {\em Einstein metrics and complex
  singularities}, Invent. Math., 156 (2004), pp.~405--443.

\bibitem{admchrus}
{\sc P.~Chru{\'s}ciel}, {\em Boundary conditions at spatial infinity from a
  {H}amiltonian point of view}, in Topological Properties and Global Structure
  of Space-Time ({E}rice, 1985), vol.~138 of NATO Adv. Sci. Inst. Ser. B Phys.,
  Plenum, New York, 1986, pp.~49--59.
\newblock Digitized version available at {\tt \footnotesize
  http://homepage.univie.ac.at/piotr.chrusciel/scans/index.html}.

\bibitem{gromsym}
{\sc M.~Gromov}, {\em Pseudoholomorphic curves in symplectic manifolds},
  Invent. Math., 82 (1985), pp.~307--347.

\bibitem{heinleb}
{\sc H.-J. Hein and C.~LeBrun}, {\em Mass in {K}{\"a}hler geometry}, Comm.
  Math. Phys., 347 (2016), pp.~183--221.

\bibitem{hirzebruchjung}
{\sc F.~Hirzebruch}, {\em \"{U}ber vierdimensionale {R}iemannsche {F}l\"{a}chen
  mehrdeutiger analytischer {F}unktionen von zwei komplexen
  {V}er\"{a}nderlichen}, Math. Ann., 126 (1953), pp.~1--22.

\bibitem{huilmpen}
{\sc G.~Huisken and T.~Ilmanen}, {\em The {R}iemannian {P}enrose inequality},
  Internat. Math. Res. Notices,  (1997), pp.~1045--1058.

\bibitem{lpa}
{\sc C.~LeBrun}, {\em Counter-examples to the generalized positive action
  conjecture}, Comm. Math. Phys., 118 (1988), pp.~591--596.

\bibitem{lp}
{\sc J.~Lee and T.~Parker}, {\em The {Y}amabe problem}, Bull. Am. Math. Soc.,
  17 (1987), pp.~37--91.

\bibitem{lerman}
{\sc E.~Lerman}, {\em Symplectic cuts}, Math. Res. Lett., 2 (1995),
  pp.~247--258.

\bibitem{liliu2}
{\sc T.-J. Li and A.-K. Liu}, {\em The equivalence between {${\rm SW}$} and
  {${\rm Gr}$} in the case where {$b\sp +=1$}}, Internat. Math. Res. Notices,
  (1999), pp.~335--345.

\bibitem{mcrules}
{\sc D.~McDuff}, {\em The structure of rational and ruled symplectic
  {$4$}-manifolds}, J. Amer. Math. Soc., 3 (1990), pp.~679--712.

\bibitem{mcmori}
\leavevmode\vrule height 2pt depth -1.6pt width 23pt, {\em Immersed spheres in
  symplectic {$4$}-manifolds}, Ann. Inst. Fourier (Grenoble), 42 (1992),
  pp.~369--392.

\bibitem{mcsal}
{\sc D.~McDuff and D.~Salamon}, {\em Introduction to {S}ymplectic {T}opology},
  Oxford University Press, New York, 1995.

\bibitem{moser}
{\sc J.~Moser}, {\em On the volume elements on a manifold}, Trans. Amer. Math.
  Soc., 120 (1965), pp.~286--294.

\bibitem{penineq}
{\sc R.~Penrose}, {\em Naked singularities}, Ann. New York Acad. Sci., 224
  (1973), pp.~125--134.
\newblock Sixth Texas Symposium on Relativistic Astrophysics.

\bibitem{symass}
{\sc R.~Schoen and S.~T. Yau}, {\em Incompressible minimal surfaces,
  three-dimensional manifolds with nonnegative scalar curvature, and the
  positive mass conjecture in general relativity}, Proc. Nat. Acad. Sci.
  U.S.A., 75 (1978), p.~2567.

\bibitem{syaction}
{\sc R.~M. Schoen and S.~T. Yau}, {\em Complete manifolds with nonnegative
  scalar curvature and the positive action conjecture in general relativity},
  Proc. Nat. Acad. Sci. U.S.A., 76 (1979), pp.~1024--1025.

\bibitem{isuvaina}
{\sc I.~{\c{S}}uvaina}, {\em A{LE} {R}icci-flat {K}\"{a}hler metrics and
  deformations of quotient surface singularities}, Ann. Global Anal. Geom., 41
  (2012), pp.~109--123.

\bibitem{taubes}
{\sc C.~H. Taubes}, {\em The {S}eiberg-{W}itten invariants and symplectic
  forms}, Math. Res. Lett., 1 (1994), pp.~809--822.

\bibitem{taubes3}
\leavevmode\vrule height 2pt depth -1.6pt width 23pt, {\em The
  {S}eiberg-{W}itten and {G}romov invariants}, Math. Res. Lett., 2 (1995),
  pp.~221--238.

\bibitem{thurston-orb}
{\sc W.~P. Thurston}, {\em The geometry and topology of $3$-manifolds}.
\newblock Unpublished Princeton Lecture Notes, available online at {\tt
  http://library.msri.org/books/gt3m/}, 1980.

\bibitem{whitney}
{\sc H.~Whitney}, {\em Differentiable manifolds}, Ann. of Math. (2), 37 (1936),
  pp.~645--680.

\bibitem{witmass}
{\sc E.~Witten}, {\em A new proof of the positive energy theorem}, Comm. Math.
  Phys., 80 (1981), pp.~381--402.

\bibitem{epwright}
{\sc E.~P. Wright}, {\em Quotients of gravitational instantons}, Ann. Global
  Anal. Geom., 41 (2012), pp.~91--108.

\end{thebibliography}
  \end{document}